\documentclass[aip,reprint,onecolumn]{revtex4-1}

\draft 

\usepackage{mathrsfs}
\usepackage{amsfonts}
\usepackage{amsthm}
\usepackage{amssymb}
\usepackage{amsmath}
\usepackage[mathscr]{eucal}
\usepackage{graphicx}
\usepackage{float}
\usepackage{subfigure}
\usepackage{xspace}
\usepackage{tikz}

\usepackage[pagewise]{lineno}

\newcommand{\p}{\partial}
\newcommand{\del}{{\partial}}
\newcommand{\dd}{\mathrm{d}}

\newcommand{\R}{\mathbb{R}}
\newcommand{\ti}{\textsf{I}_\Omega}

\newcommand{\D}{\mathbb{\D}}

\newcommand{\LL}{\mathcal{L}}

\renewcommand{\O}{{\mathcal O}}

\newcommand{\tr}{\tilde \rho}

\newcommand{\tx}{\tilde{x}}

\newcommand{\tu}{\tilde{u}}

\newcommand{\tp}{\tilde{p}}

\newtheorem{theorem}{Theorem}[section]

\newtheorem{lemma}{Lemma}[section]

\newtheorem{definition}{Definition}[section]
\newtheorem{remark}{Remark}[section]

\begin{document}


\title{High Mach number limit of one-dimensional piston
problem for non-isentropic compressible Euler equations: Polytropic gas} 

\author{Aifang Qu}
\email[ E-Mail: ]{afqu@shnu.edu.cn, aifangqu@163.com}
\thanks{The research of Aifang Qu is supported by the National Natural Science Foundation of China (NNSFC) under Grant Nos. 11571357, 11871218.}
\affiliation{Department of Mathematics, Shanghai Normal University, Shanghai  200234,  China}

\author{Hairong Yuan}
\email[Corresponding author.  E-Mail: ]{hryuan@math.ecnu.edu.cn}
\thanks{
The research of Hairong
Yuan is supported by NNSFC under Grant Nos. 11371141, 11871218, and by Science and Technology Commission of Shanghai Municipality (STCSM) under Grant No. 18dz2271000.}
\affiliation{School of Mathematical Sciences and Shanghai Key Laboratory of Pure Mathematics and Mathematical Practice,
East China Normal University, Shanghai
200241, China}

\author{Qin Zhao}
\email[ E-Mail: ]{zhao@sjtu.edu.cn}
\affiliation{School of Mathematical Sciences, Shanghai Jiao Tong University, Shanghai 200240, China}


\date{\today}

\begin{abstract}
We study high Mach number limit of the one dimensional piston problem for the full compressible Euler equations of polytropic gas, for both cases that the piston rushes into or recedes from the uniform still gas, at a constant speed. There are two different situations, and one needs to consider  measure solutions of the Euler equations to deal with concentration of mass on the piston, or formation of vacuum. We formulate the piston problem in the framework of Radon measure solutions, and show its consistency by proving that the integral weak solutions of the piston problems converge weakly in the sense of measures to (singular) measure solutions of the limiting problems, as the Mach number of the piston increases to infinity.
\end{abstract}


\maketitle 


\section{{Introduction}}\label{sec1}
Piston problem, as a prototype in the theory of mathematical gas dynamics and systems of hyperbolic conservation laws, has been studied extensively in the literature (see, for example, \cite{CWZ,D,DL2,DL,DKZ,S} and references therein). The problem may be described as follows. Suppose there is a piston which can move leftward or rightward in an infinite long and thin tube extending along the horizontal $x$-axis. The tube is filled with gas, which is enclosed by the piston on the right hand side and is uniform and still initially.  As we know, if the piston rushes into the gas with a constant speed, a shock wave will appear ahead of the piston, while a rarefaction wave will be present in front of the piston if it recedes from the gas. We are wondering what happens if the piston moves so fast, in the sense that the Mach number of the piston ({\it i.e.}, the ratio of the speed of the piston to the sound speed of the still gas) goes to infinity? Is there any limiting solution to the Euler equations and what does it look like?

To our knowledge, such a problem has not been investigated before. It turns out that there are two different cases, and for certain situation, we need to extend the admissible class of solutions to the Radon measures, since concentration of mass or formation of vacuum may appear. We propose a way to formulate the piston problem if the unknowns are general Radon measure on space-time domains. To demonstrate its reasonability, we also prove that the standard self-similar integral weak solutions of the piston problems converge weakly as measures to the limiting measure solutions.

We believe that the results are useful. In our studies of hypersonic limit of steady compressible Euler flows passing wedges \cite{QYZ}, we had discovered that one need singular measure solutions of Euler equations to put a rigorous mathematical foundation to the so called ``Newton theory" of infinite thin shock layers of hypersonic flows passing a straight wedge. It is well known that for hypersonic flow past a thin body, the problem might be simplified to a piston problem \cite{LO}. This turns our attention to the piston problem. These studies show that considering general measure solutions is necessary for a complete theory of the compressible Euler equations of polytropic gas. Also, once we have a limiting solution, it might be used for theoretical and numerical studies of piston problem when its Mach number is quite large (but not infinite), as done by Hu \cite{H}, Hu and Zhang \cite{HZ} in the studies of high Mach number supersonic flows past curved obstacles.

We review that singular measure solutions, mainly delta shocks, have been found necessary to solve compressible Euler equations of pressureless flows, or for Chaplygin gases (see, for example,  \cite{CSW2, CL, CL2, GLY, J, NS, SWY, YS,YW} and references therein), or many other hyperbolic systems whose characteristics are all linearly degenerate \cite{YZ}. In \cite{CSW}, Cavalletti {\it et.al.} showed existence of a kind of measure solution to the Cauchy problem of multi-dimensional compressible Euler equations. Huang and Wang \cite{HW} established a well-posedness theory of one-dimensional pressureless Euler equations with Radon measure as initial data. These studies focus on initial-value problems, and indicate as well the importance of considering measure solutions of systems of conservation laws.

However, it seems that there is no work on measure solutions of initial-boundary-value problems of conservation laws \cite{QYZ}. Also, the singular measure solution appeared in our work is not a delta shock, since the singular part  is supported on the boundary of the domain. We also discover that the high Mach number limit is not simply the vanishing pressure limit which was widely studied \cite{CL, CL2,GLY,SWY,YS}.

The paper is organized as follows. In Section \ref{sec2}, we formulate the piston problem in a way convenient for studying high Mach number limit, and present a definition of measure solutions. The main results are Theorem \ref{thm41} and Theorem \ref{thm51} given at the end of this section. It shall be noted that there are two situations of high Mach number limit after nondimensionalization. One for which the temperature of the gas decreases to zero, while adiabatic exponent $\gamma>1$ of the gas is fixed. The other is for which temperature being positive and fixed, while $\gamma\downarrow1$. It turns out that the latter case is more singular.
In Section \ref{sec3}, we study the case that the piston rushes to the gas and construct its integral weak solution first. Then by passing to the limit in the vague topology of measures, we find a limiting solution containing a weighted Dirac measure supported on the piston for the second case.  This also justified the concept of measure solutions we proposed.  In Section \ref{sec4}, we study the case that the piston recedes from the gas and construct self-similar solutions containing rarefaction waves. We find that the solutions converge  weakly as measures to that of the pressureless Euler flows containing a contact discontinuity and vacuum, for both cases.

\section{High Mach number limit piston problem and main results}\label{sec2}
In this section we formulate the high Mach number limit piston problem. Firstly by invariance of the compressible Euler equations under Galilean transforms, we turn to a coordinate system that moves with the piston. Then by change of independent and dependent variables, we show the piston problem depends only on two parameters, namely, $\gamma$, the adiabatic exponent of polytropic gas, and $E_0$, the (re-scaled) total energy per unit mass of the gas which is initially ahead of the piston. There are two cases for the high Mach number limit: (i) $E_0\to1/2$ for fixed $\gamma>1$; (ii) $\gamma\to 1$ for fixed $E_0>1/2$. Then we present a definition of measure solutions to the problems, and the main theorems of this paper. The proofs are given in the following two sections.

\subsection{The piston problem}\label{sec21}
The one-dimensional unsteady full compressible Euler system consists of the following conservation of mass, momentum, and energy:
\begin{equation}\label{eqeuler}
\begin{cases}
\displaystyle \del_t\rho+ \del_x(\rho u)=0,\\[8pt]
\displaystyle \del_t(\rho u)+\del_x(\rho u^2+p)=0,\\[8pt]
\displaystyle \del_t(\rho E)+\del_x(\rho u E+up)= 0.
\end{cases}
\end{equation}
Here $t\ge0$ represents time, $x\in\R$ is the space variable. The unknowns $\rho$,  $p$, $u$ and $E=\frac{1}{2}u^2+e$  are respectively the density of mass, scalar pressure, velocity, and total energy per unit mass of the polytropic gas, with
\begin{equation}\label{eqe}
e=\frac{1}{\gamma-1}\frac{p}{\rho}
\end{equation}
being the internal energy, which is proportional to the temperature of the gas.
The local sound speed of the gas is
\begin{equation}\label{eqc}
  c=\sqrt{\frac{\gamma p}{\rho}}=\sqrt{\gamma(\gamma-1)e}.
\end{equation}

Suppose initially the piston lies at $x=0$, and the gas fills the space $\{x\le0\}$. We assume the gas is static and uniform, with given state
\begin{equation}\label{eqid}
  U_\infty=(\rho,u,p)|_{t=0}=(\rho_\infty, 0, p_\infty),
\end{equation}
and the piston moves with a given constant speed $V_\infty$. Then the space-time domain we consider is
\begin{equation}\label{eqd1}
\Omega=\{(t,x): x<V_\infty t,\ t>0\}.
\end{equation}
On the trajectory of the piston
$$P=\{(t,x): x=V_\infty t,\ t\ge0\},$$
we impose the impermeable condition
\begin{equation}\label{eqbc}
u(t,x)=V_\infty \quad {\rm on}~P.
\end{equation}
Problem \eqref{eqeuler}\eqref{eqid}\eqref{eqbc} in the domain $\Omega$ is called {\it Problem (A)} below.

We define the Mach number $M_\infty$ of the piston with respect to the gas is
\begin{equation}\label{eqm1}
M_\infty=\frac{|V_\infty|}{c_\infty},
\end{equation}
with $c_\infty=\sqrt{\gamma p_\infty/\rho_\infty}=\sqrt{\gamma(\gamma-1)e_\infty}$.
The purpose of this paper is to answer the following question:

\medskip
\textit{Question:\ \
How to formulate the piston Problem (A) in the case that $M_\infty=\infty$ and construct a solution to it?}
\medskip

For the convenience of treating this question, we firstly shift the coordinates to move with the piston.
Under the following  Galilean transformation:
\begin{eqnarray*}
  &&t'=t,\ \ ~x'=x-V_\infty t,\\
  &&\rho'(t',x')=\rho(t',x'+V_\infty t'),\\
  &&u'(t',x')=u(t',x'+V_\infty t')-V_\infty,\\
  &&e'(t',x')=e(t',x'+V_\infty t'),\\
   &&p'(t',x')=p(t',x'+V_\infty t'),
\end{eqnarray*}
the equations \eqref{eqeuler} are invariant, while the domain $\Omega$ is reduced to a quarter plane
$
  \Omega'=\{(t',x'):~x'<0,~\  t'>0\},
$
and trajectory of the piston is
$$P'=\{x'=0,\  t'\ge0\}.$$

{\it For convenience of statement, we hereafter consider Problem (A) in this new coordinates and drop all the primes $``'"$ without confusion.} So the domain we consider becomes
\begin{equation}\label{eqd2}
  \Omega=\{(t,x):~x<0,~ t>0\}.
\end{equation}
The initial data is
\begin{equation}\label{eqid2}
  U_\infty=(\rho,u,E)|_{t=0}=(\rho_\infty, -V_\infty,E_\infty),
\end{equation}
and the boundary condition is
\begin{equation}\label{eqbc2}
u(t,x)=0 \quad {\rm on}~P.
\end{equation}

To understand the essence of $M_\infty\to\infty$, we carry out the following non-dimensional transformations of independent and dependent variables:
\begin{equation}
\begin{array}{cl}\label{eqT1}
  &\begin{array}{l}
   \displaystyle\tilde t=\frac{t}{T^\flat},\qquad \tx=\frac{x}{L^\flat};\\
\displaystyle\tr =\frac{\rho}{\rho_\infty},\ \ \ \ \ \ \tu =\frac{u}{|V_\infty|}, \ \
 \tp =\frac{p}{\rho_\infty V_\infty^2}, \ \ \ \tilde e =\frac{e}{V_\infty^2},
  \end{array}
\end{array}
\end{equation}
where $T^\flat$ and $L^\flat>0$ are constants with $L^\flat/T^\flat=|V_\infty|$.
Direct calculations show that
$\tilde{\rho} ,\tilde{u} ,\tilde{p} $, $\tilde{e} $ still solve
\eqref{eqeuler}, and the domain $\Omega$, boundary $P$, as well as boundary condition \eqref{eqbc2}, are invariant. However,
the initial datum become to be:
\begin{equation*}
\begin{split}
   &\tr _\infty=1,\quad \tu _\infty=\pm1, \quad  \tilde p _\infty=\frac{p_\infty}{\rho_\infty V_\infty^2}
  =\frac{1}{\gamma M_\infty^2}, \\
   &\tilde{e} _\infty=\frac{e_\infty}{V_\infty^2}= \frac{1}{\gamma-1}\frac{p_\infty}{\rho_\infty}\frac{1}{V_\infty^2}=\frac{1}{\gamma(\gamma-1)M_\infty^2},\\
   &\tilde{E}_\infty=\frac{\tu _\infty^2}{2}+\tilde{e} _\infty=\frac{1}{2}+\frac{1}{\gamma(\gamma-1)M_\infty^2}.
\end{split}
\end{equation*}%
The case $\tu _\infty=1$ means the piston rushes into the gas, while for $\tu _\infty=-1$ the piston recedes from the gas. We note that the piston problem is determined only by the parameters $\gamma$ and $M_\infty$, or equivalently, $\gamma$ and $\tilde{E}_\infty$.

{\it For simplicity of writing, in the following we still write $(\tilde{t},\tilde{x})$ as $(t,x)$, and $\tilde{U}=(\tilde{\rho}, \tilde{u}, \tilde{E})$ as $U=({\rho}, {u}, {E})$.}

In conclusion, the re-scaled piston Problem (A) consists of the Euler equations \eqref{eqeuler} defined in the domain $\Omega$, with boundary condition \eqref{eqbc2} on $P$, and
initial data
\begin{equation}\label{eqidrush}
  U_0=(\rho,u,E)(0,x)=(1,1,E_0),\quad x<0
\end{equation}
for the case the piston moves toward the gas, while
\begin{equation}\label{eqidrec}
  U_0=(\rho,u,E)(0,x)=(1,-1,E_0),\quad x<0
\end{equation}
if the piston moves backward from the gas.
By definition of $E$ and \eqref{eqe}, we have
\begin{equation}\label{eqpE}
p=(\gamma-1)\rho(E-\frac12u^2),
\end{equation}
which is considered as the state function of our problem. Henceforth we call \eqref{eqeuler}\eqref{eqbc2}\eqref{eqidrush}\eqref{eqidrec}\eqref{eqpE} as {\it Problem (B)}.

\subsection{High Mach number limit}\label{sec22}
From \eqref{eqpE} and definition of sound speed $c$,  we have
\begin{equation}\label{eqE0}
  \frac{1}{M_0^2}=\gamma(\gamma-1)(E_0-\frac{1}{2}),
\end{equation}
which implies that \begin{enumerate}
              \item[{\bf Case 1}] For fixed $\gamma>1$, the high Mach number limit $M_0\to\infty$ is equivalent to the total energy $E_0\to1/2$, or the temperature of the gas decreases to zero;
              \item[{\bf Case 2}] For fixed initial total energy $E_0>1/2$, $M_0\to\infty$ is equivalent to $\gamma\to1$.
            \end{enumerate}
For both cases, $M_0\to\infty$ implies the pressure $p_0$ of the gas which is initially ahead of the piston vanishes. However, we will see that the high Mach number limit is not always the vanishing pressure limit.

\subsection{Measure solutions to Problem (B)}
We use
$$\langle m, \phi\rangle=\int_{\bar{\Omega}}\phi(t,x)m(\dd x\dd t)$$
to write the pairing between a Radon measure $m$ supported on $\bar{\Omega}$ (the closure of $\Omega$) and a test function $\phi\in C_0(\R^2)$, which is continuous on $\R^2$ with compact support. The Lebesgue measure on $\R^d$ is denoted by $\LL^d$. For two measures $\mu$ and $\nu$,  $\mu\ll\nu$ means  $\mu$ is absolute continuous with respect to $\nu$ ({\it i.e.}, $\mu$ vanishes on $\nu$-null set). A typical Radon measure singular to the Lebesgue measure $\LL^2$ is the following Dirac measure supported on a rectifiable curve:
\begin{definition}
\label{def31}
Let $L\subset\R^2$ be a Lipschitz curve given by $\{(t,x):~x=x(t), \ \ t\in[0,T)\}$, and $w_L(t)\in L_{\mathrm{loc}}^1(0,T)$. The Dirac measure supported on $L$ with weight $w_L$ is defined by
\begin{eqnarray}\label{eq31}
\langle w_L\delta_L, \phi\rangle=\int_0^Tw_L(t)\phi(t, x(t))\sqrt{x'(t)^2+1}\,\dd t,\quad \forall \phi\in C_0(\R^2).
\end{eqnarray}
\end{definition}

Now we reformulate Problem (B) so that the unknowns may be measures, rather than functions considered before.

\begin{definition}\label{def32}
For any $E_0\ge1/2,\ \gamma\ge 1$ and $0<M_0\leq\infty$, let $\varrho, m, n, m^1, n^1,n^2, \wp$ be Radon measures supported on $\overline{\Omega}$, $w_p$ a locally integrable function on $(\R^+\cup\{0\}, \LL^1)$, and $u, E$ are $\varrho$-measurable functions. Then $(\varrho, u, E)$ is called a measure solution to the piston Problem (B),  provided that:
\begin{itemize}
\item[i)]  For any $\phi\in C_0^1(\R^2)$, there hold
\begin{eqnarray}
&&\langle \varrho, \p_t\phi\rangle+ \langle m, \p_x\phi\rangle+\int^0_{-\infty}\rho_0\phi(0,x) \dd x=0,\label{eqms1}\\
&&\langle m, \p_t\phi\rangle+ \langle n, \p_x\phi\rangle+\langle \wp, \p_x\phi\rangle-\langle
w_p\delta_{P},\phi \rangle+\int^0_{-\infty}(\rho_0u_0)\phi(0,x) \dd x=0,\label{eqms2}\\
&&\langle m^1, \p_t\phi\rangle+ \langle n^1, \p_x\phi\rangle+\langle n^2, \p_x\phi\rangle+\int^0_{-\infty}(\rho_0E_0)\phi(0,x) \dd x=0;\label{eqms3}
\end{eqnarray}

\item[ii)] $\varrho$ is nonnegative,
$m\ll \varrho$, $n\ll m$, $m^1\ll \varrho$, $n^1\ll m^1$, $n^2\ll \wp$, $\wp\ll \varrho$, and
\begin{eqnarray}\label{eq38}
u=\frac{m(\dd x\dd t)}{\varrho(\dd x\dd t)}=\frac{n(\dd x\dd t)}{m(\dd x\dd t)}=\frac{n^1(\dd x\dd t)}{m^1(\dd x\dd t)}=\frac{n^2(\dd x\dd t)}{\wp(\dd x\dd t)},\quad
E=\frac{m^1(\dd x\dd t)}{\varrho(\dd x\dd t)},
\end{eqnarray}
while $u=0, E=0$ on those sets where $\varrho=0$;
\item[iii)] If $\varrho \ll \LL^2$ with derivative $\rho(t,x)$ in a neighborhood $\O$ of $(t,x)\in\bar{\Omega}$, and $\wp \ll \LL^2$ with derivative $p(t,x)$ in $\O$, then $\LL^2$-a.e. there holds
    \begin{eqnarray}\label{eq312}
    p=(\gamma-1)\rho(E-\frac12 u^2)\qquad\text{in}\ \ \O;
    \end{eqnarray}
    In addition, the classical Lax entropy condition is valid for discontinuities of the functions $\rho, u, E$ in this situation.
\end{itemize}
\end{definition}

\begin{remark}
We may write  \eqref{eqms1}-\eqref{eqms3} formally as
\begin{eqnarray}\label{eq223}
\begin{cases}
\p_t\varrho+\p_xm=0,\\
\p_tm+\p_x(n+\wp)=-w_p\delta_{P},\\
\p_tm^1+\p_x(n^1+n^2)=0.
\end{cases}\end{eqnarray}
Note that to study boundary-value problems, the momentum equation, $i.e.$, $\eqref{eq223}_2$, contains a Dirac measure on the right hand side.
\end{remark}

\begin{definition}\label{def23}
Let $\{(\varrho_k, u_k, E_k)\}_{k=1}^\infty$ and $(\varrho, u, E)$  be measure solutions, with corresponding Radon measures $\varrho_k, m_k,  n_k, $ $m^1_k,  n^1_k, n^2_k, \wp_k$ and
$\varrho, m,  n,  m^1,  n^1, n^2, \wp$, and weights ${w_p}^k, w_p$. We say that $\{(\varrho_k, u_k, E_k)\}_{k=1}^\infty$ converges weakly as measures to $(\varrho, u, E)$  for $k\to \infty$,  if the measures $\varrho_k, m_k,  n_k,  m^1_k,  n^1_k, n^2_k, \wp_k$ converge weakly to $\varrho, m, $ $ n, $ $ m^1, $  $ n^1, $ $n^2, $$\wp$ correspondingly, and $w_p^k\to w_p$ locally in $L^1(\R^+\cup\{0\}),$ as $k\to\infty.$
\end{definition}

We can now state the main results of this paper.

\begin{theorem}\label{thm41}
For the piston moving toward the gas, Problem (B) has measure solutions for each Mach number $0<M_0\le\infty$ in the sense of Definition \ref{def32}.
\begin{itemize}
\item[\underline{Case 1}] For fixed $\gamma>1$, these solutions converge to \eqref{eqlimit1} as $E_0\to1/2$ in the sense of weak convergence of measures;

\item[\underline{Case 2}] For fixed $E_0>1/2$ and $\gamma\to 1$, these solutions converge weakly as measures to a singular measure solution with density containing a weighted Dirac measure on the piston, given by  \eqref{eqms}.
\end{itemize}
\end{theorem}

\begin{theorem}\label{thm51}
For the piston receding from the gas, for fixed $0<M_0<\infty$, Problem (B) has an integral weak solution consisting of rarefaction waves and constant states. Vacuum occurs ahead of the piston if $M_0>{2}/({\gamma-1})$. Otherwise, there is no vacuum.
\begin{itemize}
\item[\underline{Case 1}] If $\gamma>1$ is fixed, the high Mach number limit ($E_0=1/2$) corresponds to the pressureless Euler flow given by \eqref{eq46} with vacuum, and vacuum occur even for large $M_0$ (small $E_0-1/2$);

\item[\underline{Case 2}] If  $E_0>1/2$ is fixed, the high Mach number limit ($\gamma=1$) corresponds to the pressureless Euler flow given by \eqref{eq413} with vacuum, but there is no vacuum for large $M_0$ (or small but nonzero $\gamma-1$).
\end{itemize}
For both cases, the sequence of integral weak solutions converge weakly to the limiting solutions in the sense of measures.
\end{theorem}

\section{High Mach number limit for piston rushing to gas}\label{sec3}
In this section, we prove Theorem \ref{thm41}. We first construct the well-known self-similar solutions by using Rankine-Hugoniot conditions, and then taking limit of these solutions to obtain a measure solution to the limiting situation $M_0=\infty$.

\subsection{Self-similar solutions containing a shock}\label{sec31}
Since shocks appear ahead of the piston when it moves into the gas, we need the following well-known notion of integral weak solution of Problem (B).
\begin{definition}\label{def21}
We call $(\rho,u,E)\in L^\infty(\Omega;\R^3)$ an integral weak solution to Problem (B), if for any $\phi\in C_0^1(\R^2)$, there hold
\begin{equation}\label{eqdws}
 \begin{cases}\displaystyle
     \int_\Omega (\rho\p_t\phi+\rho u \p_x\phi) \dd x\dd t+\int^0_{-\infty} \rho_0(x)\phi(0,x) \dd x=0, \\
     \begin{split}
      \int_\Omega (\rho u\p_t\phi+(\rho u^2+p)\p_x\phi) \dd x\dd t&-\int_0^\infty p(t,0)\phi(t,0) \dd t+\int^0_{-\infty} \rho_0(x)u_0(x)\phi(0,x) \dd x=0,
     \end{split}\\
      \displaystyle\int_\Omega (\rho E\p_t\phi+(\rho u E+up)\p_x\phi) \dd x\dd t+\int^0_{-\infty} \rho_0(x)E_0(x)\phi(0,x) \dd x=0,
 \end{cases}
\end{equation}
and  the Lax entropy condition is valid for any discontinuities in $\Omega$.
\end{definition}

In order to find the limiting solution, we would like to obtain the solution for fixed $M_0<\infty$ first. Noting that Problem (B) is a  Riemann problem with boundary condition \eqref{eqbc2}, one can construct a piecewise constant self-similar solution of the form
\begin{equation}\label{eqshock}
U(t,x)=V(\frac{x}{t})=\begin{cases}
V_0=(1,1,E_0),& -\infty\le\frac xt<{\sigma},\\
V_1=(\rho_1,0,E_1),& {\sigma}<\frac xt\le0,
\end{cases}\end{equation}
with $\rho_1>1$ to fulfill the entropy condition. Then \eqref{eqdws} is reduced to the following Rankine-Hugoniot conditions:
\begin{equation}\label{eqrh}
  \begin{cases}
     \sigma (\rho_1-\rho_0)=\rho_1 u_1-\rho_0u_0, \\
        \sigma(\rho_1u_1-\rho_0u_0)=\rho_1 u_1^2+p_1-\rho_0u_0^2-p_0, \\
       \sigma(\rho_1E_1-\rho_0E_0)=\rho_1 u_1E_1+u_1p_1-\rho_0u_0E_0-u_0p_0.\\
  \end{cases}
\end{equation}
In view of  $\rho_0=1$, $u_0=1$, $u_1=0$, it follows from $\eqref{eqrh}_1$ that the speed of shock-front is
\begin{equation}\label{eqsigma}
  \sigma=-\frac{1}{\rho_1-1}.
\end{equation}
Inserting it into $\eqref{eqrh}_2$ gives
\begin{equation}\label{eqp1}
  p_1=p_0+1+\frac{1}{\rho_1-1}.
\end{equation}
Substituting \eqref{eqsigma} into $\eqref{eqrh}_3$ yields
\begin{equation}\label{eqE1}
  E_1=\frac{1}{\rho_1}((\rho_1-1)(E_0+p_0)+E_0).
\end{equation}
Since $p_1=(\gamma-1)\rho_1(E_1-\frac{1}{2}u_1^2)$ and $u_1=0$, we  have
$$ p_0+1+\frac{1}{\rho_1-1}=(\gamma-1)((\rho_1-1)(E_0+p_0)+E_0), $$
or equivalently,
$$ (\gamma-1)(E_0+p_0)(\rho_1-1)^2+((\gamma-1)E_0-p_0-1)(\rho_1-1)-1=0. $$
In view of $\rho_1>\rho_0=1$, it follows that
\begin{equation}\label{eqrho1}
  \rho_1=1+\frac{p_0+1-(\gamma-1) E_0+\sqrt{((\gamma-1) E_0-p_0-1)^2+4(\gamma-1)(E_0+p_0)}}{2(\gamma-1)( E_0+p_0)}.
\end{equation}

We thus obtain an integral weak solution to Problem (B) of the form \eqref{eqshock}, with $\rho_1$, $\sigma$ and $E_1$ given by \eqref{eqrho1}\eqref{eqsigma} and \eqref{eqE1} respectively. These solutions depend on the Mach number of the piston $M_0$ as mentioned in \eqref{eqE0}.
For $0<M_0<\infty$, by Definition \ref{def21}, one can easily check
that
$\varrho=\rho(t,x;M_0)\LL^2, u=u(t,x;M_0), E=E_0(t,x;M_0)$, with  $(\rho,u,E)$ the solution constructed in \eqref{eqshock}, is a measure solution to  Problem (B), just by  taking
\begin{eqnarray}\label{eq313}\begin{cases}
\varrho=\rho(t,x;M_0) \LL^2,\quad m=\rho u(t,x;M_0)\LL^2,\\ n=\rho u^2(t,x;M_0) \LL^2,\quad
m^1=\rho E(t,x;M_0) \LL^2,\\
n^1=\rho uE(t,x;M_0)\LL^2,\quad n^2=u p(t,x;M_0)\LL^2,\\
\wp=p(t,x;M_0)\LL^2,\quad w_p=w_p(t;M_0)=p_1.
\end{cases}\end{eqnarray}

In the following we consider the high Mach number limit  $M_0= \infty$ for Problem (B). As shown in Section \ref{sec22}, there are two cases.

\subsection{Case 1: Fix $\gamma>1$ and $M_0\to \infty$}

\begin{lemma}\label{lem22}
For fixed $\gamma>1$, as functions of $M_0$, all $p_0$, $\rho_1, \sigma, e_1, E_1, p_1$ are $C^2$ with respect to $M_0$, and
\begin{equation}\label{eqlimit}
  \begin{split}
      &\lim_{M_0\to\infty }p_0 =0,\quad \lim_{M_0\to\infty }E_0 =\frac12, \quad \lim_{M_0\to\infty }\rho_1 =\frac{\gamma+1}{\gamma-1},\\
& \lim_{M_0\to\infty }p_1 =\frac{\gamma+1}{2},\quad \lim_{M_0\to\infty }E_1 =\frac{1}{2},\quad \lim_{M_0\to\infty }\sigma =\frac{1-\gamma}{2}.
  \end{split}
\end{equation}
\end{lemma}

\begin{proof}
Since $p_0={1}/{(\gamma M_0^2)}$, we have $\lim_{M_0\to\infty }p_0 =0$.  That  $\lim_{M_0\to\infty }E_0 =1/2$  follows from \eqref{eqE0}.
Inserting these into \eqref{eqrho1}, one gets
$$ \lim_{M_0\to\infty }\rho_1 =1+\frac{1-\frac{\gamma-1}{2}+\sqrt{(\frac{\gamma-3}{2})^2 +2(\gamma-1)}}{\gamma-1} 
=\frac{\gamma+1}{\gamma-1}. $$
Then by \eqref{eqp1},
\begin{eqnarray*}
&&\lim_{M_0\to\infty }p_1 =\lim_{M_0\to\infty }p_0+1+\lim_{M_0\to\infty }\frac{1}{\rho_1-1}=\frac{\gamma+1}{2},\\
&&\lim_{M_0\to\infty }E_1 =\lim_{M_0\to\infty }\frac{p_1}{(\gamma-1)\rho_1}=\frac{1}{2},\\
&&\lim_{M_0\to\infty }\sigma =\lim_{M_0\to\infty }\frac{-1}{\rho_1-1}=\frac{1-\gamma}{2}.
\end{eqnarray*}
This completes proof of the lemma.
\end{proof}

It is easy to verify that
\begin{equation}\label{eqlimit1}
  U(t,x;\infty)=(\rho,u,E)(t,x;\infty)=\begin{cases}\displaystyle
           (1,1,\frac12),&\displaystyle-\infty<\frac{x}{t}<\frac{1-\gamma}{2},\\ \displaystyle (\frac{\gamma+1}{\gamma-1},0,\frac12),&\displaystyle \frac{1-\gamma}{2}<\frac{x}{t}<0
         \end{cases}
\end{equation}
is an integral weak solution of Problem (B) when $E_0=1/2$. Let $\varrho=\rho \LL^2$, then obviously, $(\varrho,u,E)$ satisfies Definition \ref{def32} and thus is a measure solution. We note that there is no concentration of mass for this case.

Recall that we denote by $(\varrho, u, E)(t,x;M_0)$ the measure solution for given $M_0$  obtained above.
That $(\varrho, u, E)(t,x;M_0)$ converge weakly as measures to  $(\varrho, u, E)(t,x;\infty)$ follows from the fact that
for any $\phi\in C_0(\R)$, there hold, as $M_0\to\infty$, with $\eta=x/t$, that
\begin{eqnarray*}
&&\int_{-\infty}^{0}F(V(\eta;M_0))\phi(\eta)\dd \eta= F(V_0(
  M_0))\int_{-\infty}^{\sigma(M_0)}\phi(\eta)\dd\eta+ F(V_1(
  M_0))\int_{\sigma(M_0)}^{0}\phi(\eta)\dd\eta\\
  &\to&  F(V_0(
  \infty))\int_{-\infty}^{\sigma(\infty)}\phi(\eta)\dd\eta+ F(V_1(
  \infty))\int_{\sigma(\infty)}^{0}\phi(\eta)\dd\eta,
  \end{eqnarray*}
and
  \begin{eqnarray*}
  &&\int_{-\infty}^{0}G(V(\eta;M_0))\phi(\eta)\dd \eta= G(V_0(\eta; M_0))\int_{-\infty}^{\sigma(M_0)}\phi(\eta)\dd\eta+ G(V_1(
   M_0))\int_{\sigma(M_0)}^{0}\phi(\eta)\dd\eta\\
  &\to&  G(V_0(\eta; \infty))\int_{-\infty}^{\sigma(\infty)}\phi(\eta)\dd\eta+ G(V_1(
  \infty))\int_{\sigma(\infty)}^{0}\phi(\eta)\dd\eta.
\end{eqnarray*}
Here $\sigma(\infty)=\lim_{M_0\to\infty} \sigma(M_0)={(1-\gamma)}/{2}$, and
\begin{eqnarray*}
&&F(V_0(\infty)) =\lim_{M_0\to\infty}(\rho_0, \rho_0u_0, \rho_0E_0)=(1,1,1/2),\\  &&F(V_1(\infty))=\lim_{M_0\to\infty}(\rho_1, \rho_1u_1, \rho_1E_1)=(\frac{\gamma+1}{\gamma-1},0, \frac{\gamma+1}{2(\gamma-1)}),  \\ &&G(V_0(\infty))=\lim_{M_0\to\infty}(\rho_0u_0, \rho_0u_0^2+p_0, \rho_0u_0E_0+u_0p_0)=(1,1,1/2),\\ &&G(V_1(\infty))=\lim_{M_0\to\infty}(\rho_1u_1, \rho_1u_1^2+p_1, \rho_1u_1E_1+u_1p_1)=(0,\frac{\gamma+1}{2},0).
\end{eqnarray*}
Then return to $(t,x)$ variables, similar to the proof of \eqref{eq43add} below,  by Definition \ref{def23}, we justified consistency of Definition \ref{def32} and proved \underline{Case 1} in Theorem \ref{thm41}.

We note that although $p_0\to0$ as $M_0\to\infty$, the limit of $p_1$ is not $0$. Thus the high Mach number limit is not the vanishing pressure limit for this case.

\subsection{Case 2: Fix $E_0>1/2$ and $M_0\to\infty$}

\begin{lemma}\label{lem22}
For fixed $E_0>1/2$,  all $p_0$, $\rho_1, \sigma, e_1, E_1, p_1$, as functions of $M_0$ (or equivalently $\gamma\ge1$), are $C^2$, and
\begin{equation}\label{eqlimit}
  \begin{split}
      &\lim_{M_0\to\infty }p_0 =0,\quad \lim_{M_0\to\infty }\rho_1 =\infty,\quad \lim_{M_0\to\infty }p_1 =1,\\
& \lim_{M_0\to\infty }E_1 =E_0,\quad \lim_{M_0\to\infty }\sigma =0,\quad \lim_{M_0\to\infty}\rho_1\sigma=-1.
  \end{split}
\end{equation}
\end{lemma}

\begin{proof}
Recall that $p_0={1}/{(\gamma M_0^2)}$, one has $\lim_{M_0\to\infty }p_0 =0$.
Inserting it into \eqref{eqrho1} and by the equivalence of $M_0\to\infty$ and $\gamma\to1$, it follows that
$$ \lim_{M_0\to\infty }\rho_1 =\lim_{\gamma\to1}\rho_1=\infty. $$
Then by \eqref{eqp1}, it holds
$$\lim_{M_0\to\infty }p_1 =\lim_{M_0\to\infty }p_0+1+\lim_{M_0\to\infty }\frac{1}{\rho_1-1}=1. $$
By \eqref{eqE1},
$$\lim_{M_0\to\infty }E_1 =\lim_{M_0\to\infty }((1-\frac{1}{\rho_1})E_0+p_0+\frac{E_0}{\rho_1})=E_0.$$
From \eqref{eqsigma}, one gets
$$\lim_{M_0\to\infty }\sigma =\lim_{M_0\to\infty }\frac{-1}{\rho_1-1}=0,$$
while
$$\lim_{M_0\to\infty }\sigma \rho_1 =\lim_{M_0\to\infty }\frac{-1}{\rho_1-1}\rho_1=-1.$$
The proof of the lemma is completed.
\end{proof}

We infer from this lemma that as $M_0\to\infty$,  the speed of the shock-front converges to the speed of the piston with an error of the order $\gamma-1$, or equivalently, $1/M_0^2$. So for this high Mach number limiting case, we may guess that the shock coincides with the piston.  To characterize clearly the limit, we encounter singular measure solutions to the Euler equations.

The idea is to take weak limit of the measure solution ({\it cf.} Definition \ref{def23}) $$\varrho=\rho(t,x;M_0)\LL^2,\  u=u(t,x;M_0), \ E=E_0(t,x;M_0)$$ obtained by \eqref{eq313}. Since the solution is self-similar, we firstly work in the variable $\eta=x/t$. For any $ \phi\in C_0(\R)$, we have
\begin{equation}
\int_{-\infty}^0\rho(\eta;M_0)\phi(\eta)\dd\eta
=\rho_0(\eta;M_0)\int_{-\infty}^{\sigma(M_0)} \phi(\eta)\,\dd\eta+\rho_1(\eta;M_0) \int_{\sigma(M_0)}^{0}\phi(\eta)\dd\eta.
\end{equation}
By \eqref{eqlimit} and recall $\rho_0=1$, we have
\begin{eqnarray*}
\lim_{M_0\to\infty}\rho_0(\eta,M_0)\int_{-\infty}^{\sigma(M_0)} \phi(\eta)\dd\eta=\int_{-\infty}^0\phi(\eta)\dd\eta,
\end{eqnarray*}
and
\begin{equation*}
  \begin{split}
  \lim_{M_0\to\infty}\rho_1(\eta;M_0) \int_{\sigma(M_0)}^{0}\phi(\eta)\dd\eta
  =&\lim_{M_0\to\infty} \rho_1(\eta;M_0))(-{\sigma(M_0)})
\lim_{M_0\to\infty}\frac{1}{-\sigma(M_0)}  \int_{\sigma(M_0)}^{0}\phi(\eta)\dd\eta\\
=&\phi(0).
  \end{split}
\end{equation*}
Therefore we proved
\begin{eqnarray*}
\lim_{M_0\to\infty}\int_{-\infty}^0\rho(\eta;M_0)\phi(\eta)\dd\eta =\int_{-\infty}^0\phi(\eta)\dd\eta+\phi(0).
\end{eqnarray*}

Similarly, we have
\begin{eqnarray*}
&&\lim_{M_0\to\infty}\int_{-\infty}^0 \rho u(\eta;M_0)\phi(\eta)\dd\eta =\int_{-\infty}^0\phi(\eta)\dd\eta,\\
&&\lim_{M_0\to\infty}\int_{-\infty}^0 \rho u^2(\eta;M_0)\phi(\eta)\dd\eta =\int_{-\infty}^0\phi(\eta)\dd\eta,\\
&&\lim_{M_0\to\infty}\int_{-\infty}^0 p(\eta;M_0)\phi(\eta)\dd\eta =0,\\
&&\lim_{M_0\to\infty}\int_{-\infty}^0 \rho  E(\eta;M_0)\phi(\eta)\dd\eta =E_0\int_{-\infty}^0\phi(\eta)\dd\eta+E_0\phi(0),\\
&&\lim_{M_0\to\infty}\int_{-\infty}^0 \rho u  E(\eta;M_0)\phi(\eta)\dd\eta =E_0\int_{-\infty}^0\phi(\eta)\dd\eta,\\
&&\lim_{M_0\to\infty}\int_{-\infty}^0  up(\eta;M_0)\phi(\eta)\dd\eta =0.
\end{eqnarray*}
Hence we proved that
\begin{equation}\label{eq42}
  \begin{split}
     &\lim_{M_0\to\infty}\int_{-\infty}^0F(V(\eta;M_0))\phi(\eta)\dd\eta=F(V_0(\infty))\int_{-\infty}^0\phi(\eta)\dd\eta+ \phi(0)(1, 0, E_0),\\
     &\lim_{M_0\to\infty}\int_{-\infty}^0G(V(\eta;M_0))\phi(\eta)\dd\eta=G(V_0(\infty))\int_{-\infty}^0\phi(\eta)\dd\eta,
  \end{split}
\end{equation}
where
\begin{equation*}
  \begin{split}
     &F(V_0(\infty))=\lim_{M_0\to\infty}(\rho_0, \rho_0u_0, \rho_0E_0)
=(1,1,E_0),\\
&G(V_0(\infty))=\lim_{M_0\to\infty}(\rho_0u_0, \rho_0u_0^2+p_0, \rho_0u_0E_0+u_0p_0)
=(1,1,E_0).
  \end{split}
\end{equation*}

Next we return to the $(t,x)$-plane. Let $\psi(t,x)\in C_0(\R^2)$ be a test function, and $\phi(t,\eta)=\psi(t,\eta t)$. Then by change-of-variables and Lebesgue dominant convergence theorem, we have
\begin{equation}\label{eq43add}
\begin{split}
\lim_{M_0\to\infty}\int_\Omega F(U(t,x;M_0))\psi(t,x) \dd x\dd t
=&\int_0^\infty t\lim_{M_0\to\infty} \int_{-\infty}^{0}F(V(\eta;M_0))\phi(t,\eta) \dd\eta \dd t\\
=&\int_0^{\infty} F(V_0(\infty))\int_{-\infty}^{0}\phi(t,\eta)t \dd\eta\dd t+(t,0, E_0t)\int_0^\infty \phi(t,0) \dd t\\
=&\int_\Omega F(V_0(\infty))\psi(t,x)\,\dd x\dd t+ (t,0, E_0t)\int_0^\infty \phi(t,0) \dd t,\\
\lim_{M_0\to\infty}\int_\Omega G(U(t,x;M_0))\psi(t,x) \dd x\dd t
=&\int_0^\infty t\lim_{M_0\to\infty} \int_{-\infty}^{0}G(V(\eta;M_0))\phi(t,\eta) \dd\eta \dd t\\
=&\int_0^{\infty} G(V_0(\infty))\int_{-\infty}^{0}\phi(t,\eta)t \dd\eta\dd t=\int_\Omega G(V_0(\infty))\psi(t,x)\,\dd x\dd t.
\end{split}
\end{equation}
These imply that all the quantities in \eqref{eq313} converge  to the following measures  in the sense of weak convergence of measures:
\begin{eqnarray}\label{eqms}\begin{cases}
\varrho=\ti\LL^2+t\delta_{P},\quad m=\ti\LL^2,\quad n=\ti\LL^2,\\
m^1=E_0\ti\LL^2+E_0t\delta_{P},\quad
n^1=E_0\ti\LL^2,\quad n^2=0,\\
\wp=0,\quad w_p(t)=1.
\end{cases}\end{eqnarray}
Here by $\textsf{I}_A$  we denote the indicator function of a set $A$, {\it i.e.},
$\textsf{I}_A(t,x)=\begin{cases}
1, &(t,x)\in A,\\
0, &(t,x)\notin A.
\end{cases}$

Now taking $M_0\to\infty$ in \eqref{eqms1}-\eqref{eqms3}, where the measures are given by \eqref{eq313}, by the above convergence results, we infer that the  measures given by \eqref{eqms} provides a singular measure solution $(\varrho=\ti\LL^2+t\delta_{P}, u=\ti, E=E_0\ti)$ to the limiting case $M_0=\infty$ for Problem (B). This also justified consistency of Definition \ref{def32}. The proof of \underline{Case 2} in Theorem \ref{thm41} is completed.

\section{High Mach number limit for piston receding from gas}\label{sec4}
This section is devoted to the proof of Theorem \ref{thm51}.
We consider the piston moves backward from the gas. As stated in Section \ref{sec2}, the initial data now is given by \eqref{eqidrec}. We can as well  construct a solution of the form $U(t,x)=V(x/t)$ and then study its limiting behaviour corresponding to the two cases of high Mach number limit.

\subsection{Self-similar solutions containing rarefaction waves}

For any fixed $0<M_0<\infty$, we suppose the solution is composed of two constant states $V_0=(1,-1,E_0)$,  $V_1=(\rho_1,0,E_1)$, and a 1-rarefaction wave $V_m$ connecting them:
\begin{equation}\label{eq222}
U(t,x)=V(\frac{x}{t})=\begin{cases}
V_0,& -\infty\le\frac{x}{t}\leq\lambda_1(V_0),\\
V_m(\frac{x}{t}),& \lambda_1(V_0)\leq\frac{x}{t}\le\lambda_1(V_1),\\
V_1,& \lambda_1(V_1)\leq\frac{x}{t}\le0,
\end{cases}\end{equation}
where $\lambda_1(U)=u-c$ is the first eigenvalue of the system \eqref{eqeuler}.
By \cite[p. 353]{S}, for given left state $\rho_0,u_0,p_0$  and right state $\rho_1,u_1,p_1$, if they can be connected by a 1-rarefaction wave, then
\begin{equation}\label{eqrws}
  \begin{cases}\displaystyle
   \frac{p_1}{p_0}=\exp{(-s)},  \\
\displaystyle    \frac{\rho_1}{\rho_0}=\exp{(-s/\gamma)},  \\
\displaystyle    \frac{u_1-u_0}{c_0}=\frac{2}{\gamma-1}\Big(1-\exp(-\frac{\gamma-1}{2\gamma}s)\Big),
  \end{cases}
\end{equation}
where $s\geq0$ is a parameter.  Therefore, we have
\begin{equation}\label{eqrwm}
  \begin{cases}\displaystyle
    p_m(s)=p_0\exp{(-s)},\\
    \displaystyle\rho_m(s)=\rho_0\exp{(-s/\gamma)},\\
   \displaystyle u_m(s)=u_0+\frac{2c_0}{\gamma-1}\Big(1-\exp(-\frac{\gamma-1} {2\gamma}s)\Big),
  \end{cases}
\end{equation}
with $s\geq0$ satisfying $u_m(s)-c_m(s)=\eta=x/t\in(-\infty, 0]$.

Recall that $\rho_0=1$, $u_0=-1$, $p_0=(\gamma-1)(E_0-\frac{1}{2})$ and thus $c_0=\sqrt{\gamma(\gamma-1)(E_0-\frac{1}{2})}$,
by $\eqref{eqrwm}_2$ we have $\exp(-s/\gamma)=\rho_m$. Inserting it into $\eqref{eqrwm}_3$ gives
\begin{equation}\label{equrho}
  u_m=-1+2\sqrt{\frac{\gamma(E_0-\frac{1}{2})}{\gamma-1}} (1-\rho_m^{\frac{\gamma-1}{2}})=-1 +\frac{2}{(\gamma-1)M_0}(1-\rho_m^{\frac{\gamma-1}{2}}).
\end{equation}
In view of
$$ p_m=p_0\exp(-s)=(\gamma-1)(E_0-\frac{1}{2})\rho_m^\gamma, $$
we have
$$ c_m=\sqrt{\frac{\gamma p_m}{\rho_m}}=\sqrt{\gamma(\gamma-1)(E_0-\frac{1}{2}) \rho_m^{\gamma-1}}.
$$
Then for the 1-rarefaction wave, the characteristic speed is
\begin{equation}\label{eqrwxi}
  \eta=u_m-c_m=-1+2\sqrt{\frac{\gamma(E_0-\frac{1}{2})}{\gamma-1}} (1-\rho_m^{\frac{\gamma-1}{2}})-\sqrt{\gamma(\gamma-1)(E_0-\frac{1}{2}) \rho_m^{\gamma-1}}.
\end{equation}
We thus could solve from this formula $\rho_m$ as a function of $\eta=x/t\in(-\infty, 0]$,  with $\gamma>1$ and $E_0>1/2$ being parameters. In this way we obtain integral weak solution $U(t,x; M_0)$ to Problem (B) for any fixed $M_0<\infty$.

We go on to see if and where vacuum will present as the piston moves away from the gas. As we know, the density decreases along the rarefaction wave as $\eta$ increases. If $\rho_m(\eta)\neq0$ when $\eta=0$, then there is no vacuum in the tube.
In this situation, if $u_1=0$, inserting it into \eqref{equrho} gives
\begin{equation}\label{eq46}
  \rho_1=\left(1-\frac{({\gamma-1})M_0}{2}\right)^{\frac{2}{\gamma-1}}.
\end{equation}
Therefore $\rho_1>0$ if and only if
\begin{equation}\label{eqrwM}
 M_0<\frac{2}{\gamma-1}.
\end{equation}
This guarantees non-existence of vacuum ahead of the piston.

If vacuum do appear in the tube, then substituting  $\rho_m=0$ into \eqref{eqrwxi} shows that the speed of vacuum boundary is
\begin{equation}\label{eqrwxivac}
\begin{split}
  \eta=-1+2\sqrt{\frac{\gamma(E_0-\frac{1}{2})}{\gamma-1}}=-1+\frac{2}{(\gamma-1)M_0}.
  \end{split}
\end{equation}
Hence the domain occupied by vacuum  is $$-1+\frac{2}{(\gamma-1)M_0}\leq\eta=\frac{x}{t}\leq0.$$

We now study the high Mach number limit.

\subsection{High Mach number limit }\label{sec42}
We firstly consider \underline{Case 1}, namely $\gamma>1$ being fixed, while $E_0\to1/2$. Taking the limit on both sides of \eqref{eqrwxivac} gives
\begin{equation}\label{eqlimitxi}
  \eta=\lim_{M_0\to\infty}(-1+\frac{3-\gamma} {(\gamma-1)M_0})=-1.
\end{equation}
Since $-1\leq u_m\leq 0$ and $0\leq\rho_m\leq1$, \eqref{eqlimitxi} means that in the limiting case the rarefaction wave degenerates to a line $x=-t$, which is a contact discontinuity, and beyond the line, the density and thus the pressure $p$ vanish. The limiting solution is given by  \begin{equation}\label{eq46}
U(t,x;\infty)=V(\frac{x}{t};\infty)=\begin{cases}
(1,-1,1/2),& -\infty\le\frac{x}{t}\le-1,\\
(0,0,0),& -1<\frac{x}{t}\le0,
\end{cases}\end{equation}
or $(\varrho=\textsf{I}_{\{x\le-t\}}\LL^2, u=-\textsf{I}_{\{x\le-t\}}, E=\frac12\textsf{I}_{\{x\le-t\}})$
It is straightforward to prove that as $M_0\to\infty$,
the solution $U(t,x; M_0)$ converges weakly to $U(t,x;\infty)$ in the sense of measures,
which is also a solution of the pressureless Euler equations:
\begin{equation}\label{eq21}
\begin{cases}
\displaystyle \del_t\rho+ \del_x(\rho u)=0,\\[8pt]
\displaystyle \del_t(\rho u)+\del_x(\rho u^2)=0,\\[8pt]
\displaystyle \del_t(\rho E)+\del_x(u\rho E)= 0,
\end{cases}
\end{equation}
which corresponds to \eqref{eq223} with $\wp=0, w_p=0$. We thus proved \underline{Case 1} claimed in Theorem \ref{thm51}.

Next we consider \underline{Case 2}, for which $E_0>1/2$ is fixed, and $\gamma\to 1$.  It follows from $E_0=\frac12+\frac{1}{\gamma(\gamma-1)M_0^2}$ ({\it cf.} \eqref{eqE0}) that
\begin{eqnarray}\label{eq412}
M_0=\frac{1}{\sqrt{\gamma(\gamma-1)(E_0-\frac12)}}.
\end{eqnarray}
Then $(\gamma-1)M_0=O({1}/{M_0})$. It means that \eqref{eqrwM} holds for sufficiently large Mach number $M_0$ and thus no vacuum appears ahead of the piston for large but finite  $M_0$.

Now from \eqref{eq412} and \eqref{equrho}, we have
$$\rho_m=\left(1-\frac{u_m+1}{2}\sqrt{\frac{\gamma-1}{\gamma(E_0-\frac12)}}\right)^{\frac{2}{\gamma-1}},$$
hence $\lim_{M_0\to\infty}\rho_m=\lim_{\gamma\to 1}\rho_m=0$ whenever $u_m>-1$. This means in the limit the gas is also rarefied to vacuum in the domain $-1<x/t\le0$, although we knew for any fixed $M_0$ large, there is no vacuum. Furthermore, we have
\begin{equation*}
  \begin{split}
     \lim_{M_0\to\infty}p_0&=0,\quad \lim_{M_0\to\infty}p_m =\lim_{M_0\to\infty}(\gamma-1)\rho_m^\gamma(E_0-\frac12)=0;\\
     \lim_{M_0\to\infty}e_m&=\lim_{M_0\to\infty} \frac{1}{ \gamma-1}\frac{p_m}{\rho_m}
     =E_0-\frac12, \quad \lim_{M_0\to\infty}c_m=0.
  \end{split}
\end{equation*}
The limiting solution might be written as
\begin{equation}\label{eq413}
U'(t,x;\infty)=V'(\frac{x}{t};\infty)=\begin{cases}
(1,-1,E_0),& -\infty\le\frac{x}{t}\le-1,\\
(0,0,0),& -1<\frac{x}{t}\le0,
\end{cases}\end{equation}
which is also a solution to the pressureless Euler equations \eqref{eq21}. The weak convergence of $U(t,x; M_0)$ to $U'(t,x;\infty)$ in the sense of measures and consistency could be checked directly. We thus proved \underline{Case 2} in Theorem \ref{thm51}.


%
%

%



\end{document}